\documentclass[10pt,leqno]{article}

\usepackage{amsmath,amssymb,amsfonts,graphicx,bm,amsthm,mathrsfs}
\usepackage[usenames]{xcolor}
\usepackage{soul}
\usepackage[colorlinks=true]{hyperref}
\usepackage{cleveref}
\usepackage{appendix}
\usepackage{pstricks}

\usepackage{enumitem}
\usepackage{latexsym}
\usepackage{stmaryrd}
\usepackage{pstricks}
\usepackage{aliascnt}

\newcommand{\Kr}{$\mathcal{K}=(K, \leq, D, \Vdash)$} 

\numberwithin{equation}{section} 
\newtheorem{theorem}{Theorem}[section]
\newaliascnt{lemma}{theorem}  
\newtheorem{lemma}[lemma]{Lemma}  
\aliascntresetthe{lemma}  
  
\newaliascnt{definition}{theorem}  
\newtheorem{definition}[definition]{Definition}  
\aliascntresetthe{definition}  
 
\newaliascnt{corollary}{theorem}  
\newtheorem{corollary}[corollary]{Corollary}  
\aliascntresetthe{corollary}  
 
\newaliascnt{proposition}{theorem}  
\newtheorem{proposition}[proposition]{Proposition}  
\aliascntresetthe{proposition}  
 
\newaliascnt{remark}{theorem}  
\newtheorem{remark}[remark]{Remark}  
\aliascntresetthe{remark}  
 
\newaliascnt{notation}{theorem}  
  
\aliascntresetthe{notation}  
 
\newaliascnt{example}{theorem}  
  
\aliascntresetthe{example}  
 
\newaliascnt{conjecture}{theorem}  
  
\aliascntresetthe{conjecture}  
 
\newaliascnt{question}{theorem}  
\newtheorem{question}[question]{Question}  
\aliascntresetthe{question}  
 
\newaliascnt{fact}{theorem}  
  
\aliascntresetthe{fact}  
 
\newaliascnt{claim}{theorem}  
  
\aliascntresetthe{claim}  
 
\def\kcal{\mathcal{K}}

\def\HA{\hbox{\sf HA}{} }

\def\PA{\hbox{\sf PA}{} }

\def\K4{\hbox{\sf K4}{} }

\def\lcal{\mathcal{L}}

\newcommand{\ra }{\rightarrow }

\usepackage{subfig}
\usepackage{color}
\usepackage{tikz}
\usetikzlibrary{positioning,arrows,calc,patterns,topaths,shapes}
\tikzset{
	modal/.style={>=stealth',shorten >=1pt,shorten <=1pt,auto,node distance=1.5cm,
		semithick},
	world/.style={circle,draw,minimum size=0.5cm,fill=gray!15},
	point/.style={circle,draw,inner sep=0.5mm,fill=black},
	reflexive above/.style={->,loop,looseness=7,in=120,out=60},
	reflexive below/.style={->,loop,looseness=7,in=240,out=300},
	reflexive left/.style={->,loop,looseness=7,in=150,out=210},
	reflexive right/.style={->,loop,looseness=7,in=30,out=330}
}
\tikzstyle{world}=[circle,draw,inner sep=0pt,minimum size=25pt]

\tikzstyle{yworld}=[circle,fill=yellow,draw,inner sep=0pt,minimum size=25pt]

\begin{document}
\setstcolor{red}
\title{Localizing Finite-Depth Kripke Models}
\author{
Mojtaba Mojtahedi\thanks{mojtaba.mojtahedi@ut.ac.ir}\\
School of Mathematics, Statistics and Computer Science\\
College of Science, University of Tehran
 }

\maketitle

\begin{abstract}
We can  look at a first-order   (or propositional)  intuitionistic Kripke model 
 as an ordered set of classical models. 
In this paper, we show that for a 
finite-depth  Kripke model in an arbitrary first-order language or propositional language, 
local (classical) truth of a formula is equivalent to non-classical truth 
(truth in the Kripke semantics) of a Friedman's translation 
of that formula, i.e.  $ \alpha\Vdash A^\rho \Leftrightarrow \mathfrak{M}_\alpha\models A$. 
We introduce some applications of this fact.  
We extend the result of \cite{AH2} and show that semi-narrow
 Kripke models of Heyting Arithmetic $ \HA $ are locally $ \PA $. 
\end{abstract}

{\em keywords:} Intuitionistic logic, Kripke models, local truth,  finite depth.

\section{Introduction}\label{Int}
D. van Dalen et al. in \cite{DKMV} introduced a very useful  technique, called {\em pruning} of a Kripke model,  for studying Kripke 
semantics of $\HA$. Their method is a correspondence between forcing of Friedman's 
translation of a sentence in a Kripke model, and forcing of that sentence in a sub-model  
(in the sense of \cite{Visser-submodels})  of the same Kripke model. 
By this method, they proved that every finite Kripke model of \HA is $\PA$-normal,  
and every $\omega$-frame Kripke model of \HA is locally $\PA$ for infinitely many nodes of the model. 
Then K.\,F.\,Wehmeier in \cite{Weh96} strengthened this result to a wider class of Kripke models, e.g., finite-depth Kripke models, and some special infinite Kripke models.
Ardeshir and Hesaam in \cite{AH2} showed that every rooted narrow 
tree-frame Kripke model of $\HA$ is locally $\PA$.
In this paper, by  iterated use of 
the pruning lemma introduced in \cite{DKMV}, we show that for 
any node $\alpha$ of a  finite depth  Kripke model,  
there exists a sentence $\rho$, such that  for all formula 
$ A $ 
\begin{center}
$\alpha\Vdash A^\rho$ if and only if  $\alpha\vDash A $,
\end{center}
where $A^\rho$ is Friedman's translation of $A$ by $\rho$.

\section{Definitions, conventions and preliminaries}\label{pre}
The propositional language $\lcal_0$ contains $\{\vee, \wedge, \ra,\bot\}$ and atomic variables $p_1,p_2,\ldots$.
The language $\lcal_1$ is the first-order language, 
i.e. as connectives contains $\{\vee, \wedge, \ra,\bot\}$, and 
quantifiers $\{\forall,\exists\}$, plus some function symbols,  
relation symbols, a special equality symbol $ = $, constant symbols and variables $ x,y,z,\dots$.
We use $\neg A$ as an abbreviation for $A\ra\bot$.  
The language of arithmetic $\lcal_{\sf a}$ contains $ \{+,.\}$, $\{=,<\}$ and $\{0,1\}$ as two function symbols, predicate symbols and  constant symbols, respectively. 
For arbitrary set $D$, we  use the notation 
$\lcal(D)$ as the language $\lcal$ augmented with the new set of constant symbols $D$.
We  use  $\vdash$ and $\vdash^c$,  for  intuitionistic and classical deductions, respectively. 

A Kripke model for  a   language $\lcal$ is a quadruples 
$\mathcal{K}=(K, \leq, D, \Vdash)$ with the following properties:
\begin{itemize}
	\item  $K$ is a non-empty set (of {\em nodes}), and  $ (K, \leq) $ is a poset,
	\item $D$ is a function (the domain function) 
	      from $K$ such that $D(\alpha)$ is non-empty, 
	\item For all $\alpha\leq\beta\in K$, we have $D(\alpha)\subseteq D(\beta)$,
	\item $\Vdash$ is a binary relation with first component in $ K$ and second component in 
the set of atomic sentences in $\lcal(D(\alpha))$,
    \item For all 
$\alpha\leq\beta$ and atomic sentence $A$ in $\mathcal{L}(D(\alpha))$, if  
$\alpha\Vdash A$ then $\beta\Vdash A$ ({\em monotonicity}).
\end{itemize}

We can extend $\Vdash$ to all sentences in the language $\lcal$  recursively, just like 
classical case, except for  $\rightarrow$ and $ \forall $ by the following items:
\begin{itemize}
\item $\alpha \Vdash A\ra B$ iff for all $\beta\geq\alpha$ , if  $\beta\Vdash A$ then 
$\beta\Vdash B$,
\item $ \alpha\Vdash \forall{x}\, A $ iff for all $\beta\geq\alpha$ and $ b\in D(\beta) $, we have  
$ \beta\Vdash A[x:b] $.
\end{itemize}

By this definition we can assign to each Kripke model {\Kr}, a triple 
$(K, \leq,\mathcal{F})$, where $\mathcal{F}(\alpha)=\mathfrak{M}_{\alpha}$ is a classical model 
for the language $\lcal(D(\alpha))$, with $D(\alpha)$ as its universe with 
the property: 
``for each atomic formula $A$,  $\mathfrak{M}_{\alpha}\models A$
iff  $\alpha\Vdash A$".
Let $T$ be a first-order theory in ${\cal L}_{\sf a}$. 
A Kripke model {\Kr} is called $T$-{\em normal} or {\em locally} $T$, if  $\mathcal{F}(\alpha)\models T$,  for all $\alpha\in K$.
In the rest of the paper, we use the notation $\alpha\models A$ instead of 
$\mathcal{F}(\alpha)\models A$.

For a fixed sentence $\rho$ and a sentence $A$ in a language $\lcal$, 
Friedman's translation of $A$ by $\rho$,  
$A^\rho$,  is defined recursively 
 by replacing all occurrences of atomic sub-formulas  of $A$
by their disjunction with $\rho$. More precisely, $A^\rho$ is defined inductively as follows:
\begin{itemize}
	\item $A^\rho:= A\vee \rho$, for atomic formula $A $, 
	\item $ (A_1\circ A_2)^\rho:= A_1^\rho\circ  A_2^\rho  $ and 
	$ \circ\in \{\vee,\wedge,\to\} $,
	\item $ (\forall x A)^\rho:=\forall x (A^\rho) $ and 
	$ (\exists x A)^\rho:=\exists x (A^\rho)$.
\end{itemize}
 We have the following facts about Friedman's 
translation (see \cite{DKMV}):
\begin{proposition}\label{friedman}
\begin{itemize}
\item $\rho\vdash A^\rho$,
\item if $\Gamma\vdash A$ then:  $\Gamma^\rho\vdash A^\rho$,
\item $\vdash^c A^\rho\leftrightarrow(A\vee\rho)$,
\item For any  $A\in\Sigma_1$,  we have $\HA\vdash A^\rho\leftrightarrow(A\vee\rho)$, 
\item $\HA\vdash\HA^\rho$ \textup{(}\HA is closed under Friedman's translation\textup{)}
\end{itemize}

\begin{lemma}\label{Lem-Friedman-neg-rho}
 		$ \neg \rho\vdash A\leftrightarrow A^\rho $.
 	\end{lemma}
 \begin{proof}
 	Use induction on the complexity of $ A $.
 \end{proof}
 
\end{proposition}
We call a node $\alpha$ in a Kripke model \Kr, {\em classical}, iff  $\alpha\Vdash\forall{\bar{x}}(A\vee\neg A)$, for all $A$ in $\lcal(D(\alpha))$,  where $\bar{x}$ are all free variables of $A$. We have the following facts from \cite{DKMV}:

\begin{proposition}\label{ClassicalNode}
 For any  Kripke models \Kr,
 \begin{itemize}
  \item the following conditions are equivalent:
   \begin{itemize}
    \item $\alpha$ is a classical node,
    \item $\alpha$ forces all $\lcal$-sentences $\forall\bar{x}(A\vee\neg A)$,
    \item For all $A\in \lcal(D(\alpha))$: $\alpha\models A$ iff $\alpha\Vdash A$,
   \end{itemize}
  \item all final nodes are classical,
  \item if  $\alpha$ is classical, then so is $\beta$ for all $\beta\geq\alpha$.
 \end{itemize}
\end{proposition}

\section{Localizing finite-depth Kripke models}

Let  \Kr  \ be a Kripke model and $\rho$ be a fixed sentence. We can define a new Kripke model, the {\em pruned} model with respect to $\rho$,   
$\mathcal{K}^\rho=(K^\rho, \leq^\rho, D^\rho, \Vdash^\rho)$, where  
$K^\rho=K\setminus\{\alpha\in{K}\mid \alpha\Vdash\rho\}$  
and $\leq^\rho$, $D^\rho$ and $\Vdash^\rho$ are restriction of  $\leq$, $D$, and $\Vdash$,  respectively 
to the set $K^\rho$. 
\begin{lemma}\label{pruning-lemma}{\bf (Pruning Lemma)}\cite{DKMV}
 Let $\rho\in\lcal$ be a sentence and  {\Kr} be a Kripke model for the language $\lcal$ and $\alpha\in K$ such that $\alpha\nVdash\rho$.  Then for all sentences $A$ in the language $\lcal(D(\alpha))$ :
\begin{center}                            
$\alpha\Vdash^\rho A$ iff $\alpha\Vdash A^\rho$.
\end{center}
\end{lemma}

 The following lemma shows that Friedman's translations are associative:
\begin{lemma}\label{FriedmanAssociative}
For all sentences $\rho_1, \rho_2$ and formula $A$:
\begin{center}
$\vdash A^{(\rho_1^{\rho_2})}\leftrightarrow (A^{\rho_1})^{\rho_2}$.
\end{center}
\end{lemma}
\begin{proof}
We prove this lemma by induction on the complexity of  $A$:\\
\begin{itemize}
 \item  $A$ is an atomic.  First note that using the first item of \Cref{friedman} implies   
 $\vdash \rho_2\to \rho_1^{\rho_2}$,  and then  
 $\vdash (\rho_2\vee \rho_1^{\rho_2})\leftrightarrow \rho_1^{\rho_2}$. Hence  
  $(A^{\rho_1})^{\rho_2}=A\vee\rho_2\vee\rho_1^{\rho_2}\leftrightarrow A\vee\rho_1^{\rho_2}=A^{(\rho_1^{\rho_2})}$. 
 \item  $A=C\circ B$ and $\circ\in\{\vee,\wedge,\ra\}$. Then  $(C\circ B)^{(\rho_1^{\rho_2})}=C^{(\rho_1^{\rho_2})}\circ B^{(\rho_1^{\rho_2})}\leftrightarrow(C^{\rho_1})^{\rho_2}\circ(B^{\rho_1})^{\rho_2}=(C^{\rho_1}\circ B^{\rho_1})^{\rho_2}=(A^{\rho_1})^{\rho_2}$.
 \item $A=Q{x}B$ and  $Q\in\{\forall,\exists\}$. Then 
  $$A^{(\rho_1^{\rho_2})}=(Q{x}B)^{(\rho_1^{\rho_2})}=Q{x}B^{(\rho_1^{\rho_2})}\leftrightarrow Q{x}(B^{\rho_1})^{\rho_2}=(Q{x}B^{\rho_1})^{\rho_2}=(A^{\rho_1})^{\rho_2} $$
\end{itemize}
\end{proof}

\begin{definition}\label{Def-PEM-star}
Let $\Gamma$ be a set of formulas. We define $\Gamma^*=\bigcup \Gamma^n$, 
in which $\Gamma^n$ is defined inductively as follows
\begin{itemize}
\item $\Gamma^0:=\{\bot\}$,
\item $\Gamma^{n+1}:=\{A^B : A\in\Gamma^n, \ B\in \Gamma\}$.
\end{itemize}
Also we define ${\sf PEM}(\lcal)$  
 as the set of the universal closures of all instances of the principle of  excluded middle $A\vee\neg A$ in the language $\lcal$.  
Moreover, ${\sf PEM}_{\sf sen}(\lcal)$ is defined as the set of   all instances of the principle of  excluded middle $A\vee\neg A$, for sentence $A$  in  $\lcal$.
When no confusion is likely, we might skip $ \lcal $ in the notation 
$ {\sf PEM}(\lcal) $ and other similar notations.
\end{definition}
Note that in the above definition $\Gamma^1$ includes  (an equivalent form of) all formulas   
$A\in\Gamma$,
 $\Gamma^2$ includes (an equivalent form of) all $A^B$, in which $A,B\in \Gamma$, 
 $\Gamma^3$ includes 
 (an equivalent form of) all $(A^{B})^C$, in which $A,B,C\in \Gamma$ and so on.  More importantly 
 $ \Gamma^* $ is closed under 
 	$\Gamma$-Friedman's  translation , i.e. for all $ A\in \Gamma^* $ 
 	and $ B\in \Gamma $, we have $ A^B\in\Gamma^* $.

\begin{definition}\label{Def-K-alpha}
\begin{enumerate}
\item	Let $ \alpha $ be a node of a  Kripke model $ \kcal $ and let 
	$\kcal_\alpha$ denotes the {\em truncated} of $\kcal$ with respect to $\alpha$, i.e., 
	restriction of $\kcal$ to 
	all nodes $\beta\geq \alpha$, with the same forcing relation for atomic formulas as $\kcal$.  
\item 	
	We define the {\em depth} of $\kcal$, indicated by $d(\kcal)$,  as the maximum natural number $ n $, such that no path in $(K, \leq)$ is longer than n. We denote 
 $d_\kcal(\alpha):=d(\kcal_\alpha)  $  or simply $d(\alpha):=d(\kcal_\alpha)  $ if no confusion is likely.
\item 	
	We also define
	 $ \kcal'_\alpha $ as the restriction of the nodes of 
	$ \kcal $ to the following set:
	$$ \{\alpha\}\cup \{\beta : \beta >  \alpha \text{  is not classical}\} \cup  
	\{\beta : \beta>\alpha \wedge \neg\exists\gamma (\gamma>\beta)  \},$$
with the same forcing relation for atomic formulas as $\kcal$.	
In other words, $ \kcal'_\alpha $ is derived from $ \kcal_\alpha $ by eliminating all  
classical  nodes which are strictly above $ \alpha $ and are not leaves.
\end{enumerate}	
\end{definition}
Now we have our main result.
\begin{theorem}\label{main}
Suppose {\Kr} is  a finite-depth Kripke model for the language $\lcal$. 
Then for any   $\alpha\in K$, there exists some $\rho\in{\sf PEM}(\lcal)^*$ 
such that for any sentences $A$ in $\lcal(D(\alpha))$,
\begin{center}
  $\alpha\Vdash A^{\rho}$ iff $\alpha\models A$.
 \end{center} 
\end{theorem}
\begin{proof}
We use induction on ${d}(\alpha)$. 
\begin{itemize}
  \item If $d(\alpha)=0$, then  $\alpha$ is terminal node (a leaf) and hence 
  by \Cref{ClassicalNode}, it is  a classical node.  
  Then we  take $\rho:=\bot\in{\sf PEM}(\lcal)^0$.
  \item  
  Suppose that we have the induction hypothesis for all {\Kr}, $\beta\in K$  with $d(\beta)<n$.
  Let {\Kr} a finite-depth Kripke model, $\alpha\in K$  and  $d(\alpha)=n>0$.  
  If $\alpha$ is a classical node, by  \Cref{ClassicalNode},  we may let $\rho:=\bot$. 
Otherwise,   \Cref{ClassicalNode} implies
  $\alpha\nVdash\forall\bar{x}(A(\bar{x})\vee\neg A(\bar{x}))$,  
  for some formula $A(\bar{x}) \in \lcal $ with free variables in $\bar{x}$. Let 
  $\tau:= \forall\bar{x}(A(\bar{x})\vee\neg A(\bar{x}))\in {\sf PEM}$.  
  Then by Pruning Lemma, for any $B$,  $\alpha\Vdash^\tau B$ iff $\alpha\Vdash B^\tau$.
  By \Cref{ClassicalNode}, we know that $d(K_\alpha^\tau)<n$, and  
  by induction hypothesis,  
  there exists some $\rho'\in{\sf PEM}(\lcal)^*$ 
  such that for all $A$, we have 
 \begin{center} 
 $\alpha\Vdash^\tau A^{\rho'}$ iff  
 $\alpha\models A$ iff  
$\alpha\Vdash (A^{\rho'} )^\tau$. 
\end{center} 
 By associativity of Friedman's translation (\Cref{FriedmanAssociative}), we have
\begin{center}  
 $\alpha\models A$ iff  
$\alpha\Vdash A^{({\rho'}^\tau)}$.
\end{center} 
Now we  define $\rho:=\rho'^\tau$.
Since 
  $\rho'\in {\sf PEM}^*=\bigcup_{i\in \mathbb{N}} {\sf PEM}^i$, there is some $ k\in \mathbb{N} $ 
  such that 
  $ \rho'\in {\sf PEM}^k$. Hence by \Cref{Def-PEM-star} $\rho =\rho'^\tau\in {\sf PEM}^{k+1}\subseteq {\sf PEM}^*$, as desired.\qedhere
\end{itemize}
\end{proof}

The above theorem   could be adopted for the propositional language as well. 
\begin{theorem}\label{Theorem-Main-Prop}
Let $\kcal=(K, \leq, \Vdash)$ be a finite-depth Kripke model for the propositional language. 
For any   $\alpha\in K$, there exists some $\rho\in{\sf PEM}_{\lcal_0}^*$ 
such that for any proposition $A$, 
\begin{center}
$\alpha\Vdash A^{\rho}$ iff $\alpha\models A$.
\end{center}
\end{theorem}

\begin{remark}
{\em A sentence $\rho\in\lcal$ is called a  {\em localizer} for some node $\alpha$ of a Kripke model for the language $\lcal$, if for  any  
sentence $A\in\lcal(D(\alpha))$,  
\begin{center} 
$\alpha\Vdash A^{\rho} \Longleftrightarrow \alpha\models A$.
\end{center} 
}
\end{remark}

In the next proposition, we show that it is {\em  not} possible to find some  localizer $\rho$ to be applied {\em uniformly}  for all Kripke models and nodes with some given height. This means that $\rho$ really depends on the Kripke model and the assigned node.  

\begin{proposition}\label{Proposition-not-uniform}
Given some number $d\geq 1$ and a first-order language $\lcal$, it is not possible to find some  localizer
$\rho$  for all  $\alpha$ in an arbitrary Kripke model with $d(\alpha)=d$. 
\end{proposition}
\begin{proof}
We prove by contradiction. For the sake of contradiction, assume some  uniform  localizer $\rho$, for all nodes with depth $d$. 
\\
{\em Claim:} $\vdash^c\neg\rho$.\\
Before   we continue with the proof of the claim, let us see how this claim finishes the proof. 
From the claim one can deduce that $\rho$ is not forced in the leaves of any Kripke model (since in leaves intuitionistic and classical truth coincide). 
Hence $\neg\rho$ is  forced in any node of any finite-depth Kripke model. \Cref{Lem-Friedman-neg-rho} 
 implies that for any node $\alpha$ of  any finite-depth Kripke model and for all sentence $A$, 
we have $\alpha\Vdash A^\rho\leftrightarrow A$. This implies that $\alpha\Vdash A$ iff $\alpha\Vdash A^\rho$. Since for any $\alpha$ with the depth $d$, we have 
$\alpha\Vdash A^\rho$ iff $\alpha\models A$,  
one may deduce  $\alpha\Vdash A$ iff $\alpha\models A$. Then it is quite straightforward  to find some Kripke model {\Kr}, $\alpha\in K$  with
$d(\alpha)=d\geq 1$ 
and $A\in\lcal(D(\alpha))$, such that it is not the case that $\alpha\Vdash A$ iff $\alpha\models A$. This contradicts  our previous result.
\\
{\em Proof of the claim:}
Assume that $\nvdash^c\neg\rho$. Then there exists some classical structure $\mathfrak{M}\models \rho$. Define a Kripke model $\kcal$ by adding $d-1$ copies 
of $\mathfrak{M}$ in beneath of $\mathfrak{M}$. Let $\alpha_0$ be the root of $\kcal$. Then clearly $d(\alpha_0)=d$ and hence for any $A\in\lcal(D(\alpha_0))$, 
 we have $\alpha_0\Vdash A^\rho$
iff $\alpha_0\models A$. Since $\mathfrak{M}\models \rho$, we have $\alpha_0\Vdash A^\rho$ 
for any  $A\in\lcal(D(\alpha_0))$. Then  for all $A\in\lcal(D(\alpha_0))$, we have 
$\mathfrak{M}\models A$. In particular, $\mathfrak{M}\models \bot$, a contradiction.
\end{proof}

What happens for infinite-depth Kripke models?
In this case, there might not exist any  localizer at all. 
Here we will present a counter-example for the propositional language.
 Since the propositional language is a special case of a first-order language, 
 this counter-example is a counter-example for the first-order language as well.
 Let $\kcal$ be any Kripke model for which 
  the propositional intuitionistic logic is complete (for example the canonical model is such a Kripke model). 
Add some node $\alpha_0$ in beneath of all other nodes of $\kcal$. 
From completeness of $ \kcal $ for the  intuitionistic propositional logic,  we have
 $\alpha_0 \Vdash A$ iff $\vdash A$, for any $A$.
We will show that $\alpha_0$ doesn't have any  localizer. 
Suppose not, i.e.  $\rho$ is  a  localizer for $\alpha_0$. 
Then  $ \alpha_0\Vdash A^\rho $ iff $\alpha_0\models A$, for any $A$. Let $A=\bot$. 
Hence $\alpha_0\nVdash \rho$. By soundness, we have  $\nvdash \rho$.
Since for atomic  $A$, we have $A^\rho=A\vee \rho$, and by disjunction property of 
intuitionistic (propositional) logic, 
and $\nvdash \rho$, 
we may deduce $ \alpha_0\models A $ iff $ \vdash A $ for atomic $ A $. 
Hence $ \alpha_0\models  A\to \bot  $, for all atomic $ A $. Then  
for all atomic $ A $, we have $ \vdash A\to \rho $. This implies that $ \vdash \rho $, a contradiction. 
\begin{remark}
	Localizers for infinite-depth nodes of Kripke models might not exist.
\end{remark}

Although localizers for infinite depth Kripke models may not exist, 
we will show that, by use of methods in\cite{AH2},
for a class of  semi-narrow Kripke models   (definition comes next), 
 which includes finite-depth and also  some infinite models, there exist some sort of localizers (\Cref{theorem3}). 
 \begin{definition}
 	A Kripke model is narrow if there is no infinite set of pairwise incomparable nodes.
 	We say that a Kripke model is semi-narrow, 
 	if  for any set of pairwise incomparable nodes  $ X $
 	  there is some $ n $ such that for almost all $ u\in X $ (all but finitely many of them),
 	   we have $ d(u)\leq n $. 
 	
 \end{definition}
Note that all finite depth Kripke models and also all narrow Kripke models are semi-narrow, but the converse is not necessarily true. For example the comb frame is 
semi-narrow and it is neither  narrow nor  finite-depth. 

\begin{figure}[ht]\centering
	\minipage{0.32\textwidth}
\scalebox{.4}{\hspace{3.5cm} \begin{tikzpicture}[modal,scale=2]
	\node[point]  (a) at (0,0) {.};
	\node[point]  (b) at (.3,1) {.};
	\node[point]  (c) at (-.3,1) {.};
	\node[point]  (d) at (.6,2) {.};
	\node[point]  (e) at (-.6,2) {.};
	\node[point]  (f) at (.9,3) {.};
	\node[point]  (g) at (-.9,3) {.};
	\node[point]  (h) at (1.2,4) {.};
	\node[point]  (i) at (-1.2,4) {.};
	\node[]  (j) at (1.5,5) {};
	\draw[-] (a) -- (b);
	\draw[-] (a) -- (c);
	\draw[-] (b) -- (d);
	\draw[-] (c) -- (e);
	\draw[-] (d) -- (f);
	\draw[-] (e) -- (g);
	\draw[-] (f) -- (h);
	\draw[-] (g) -- (i);
	\draw[-] (h) -- (j);
	\end{tikzpicture}}
\caption*{\footnotesize{``tick” frame}} 
	\endminipage\hfill
	\minipage{0.32\textwidth}
\scalebox{.4}{\hspace{3.5cm}\begin{tikzpicture}[modal,scale=2]
	\node[point]  (a) at (0,0) {.};
	\node[point]  (b) at (.3,1) {.};
	\node[point]  (c) at (-.3,1) {.};
	\node[point]  (d) at (.6,2) {.};
	\node[point]  (e) at (-.6,2) {.};
	\node[point]  (f) at (.9,3) {.};
	\node[point]  (g) at (-.9,3) {.};
	\node[point]  (h) at (1.2,4) {.};
	\node[point]  (i) at (-1.2,4) {.};
	\node[]  (j) at (1.5,5) {};
	\node[]  (k) at (-1.5,5) {};
	\draw[-] (a) -- (b);
	\draw[-] (a) -- (c);
	\draw[-] (b) -- (d);
	\draw[-] (c) -- (e);
	\draw[-] (d) -- (f);
	\draw[-] (e) -- (g);
	\draw[-] (f) -- (h);
	\draw[-] (g) -- (i);
	\draw[-] (h) -- (j);
	\draw[-] (i) -- (k);
	\end{tikzpicture}}
\caption*{\footnotesize{``V” frame}}
	\endminipage\hfill
	\minipage{0.32\textwidth}%
\scalebox{.4}{\hspace{4.5cm}\begin{tikzpicture}[modal,scale=2]
	\node[point]  (a) at (0,0) {.};
	\node[point]  (b) at (.3,1) {.};
	\node[point]  (c) at (-.3,1) {.};
	\node[point]  (d) at (.6,2) {.};
	\node[point]  (e) at (0,2) {.};
	\node[point]  (f) at (.9,3) {.};
	\node[point]  (g) at (.3,3) {.};
	\node[point]  (h) at (1.2,4) {.};
	\node[point]  (i) at (.6,4) {.};
	\node[]  (j) at (1.65,5.5) {};
	\node[point]  (k) at (.9,5) {.};
	\draw[-] (a) -- (b);
	\draw[-] (a) -- (c);
	\draw[-] (b) -- (d);
	\draw[-] (b) -- (e);
	\draw[-] (d) -- (f);
	\draw[-] (d) -- (g);
	\draw[-] (f) -- (h);
	\draw[-] (f) -- (i);
	\draw[-] (h) -- (j);
	\draw[-] (h) -- (k);
	\end{tikzpicture}}
\caption*{\footnotesize{``comb” frame}} 
	\endminipage
\end{figure}

In \cite{AH2}, it is shown that all rooted narrow Kripke models of $ \HA $ are
locally $ \PA $. Here we extend that result to the class of semi-narrow models and also 
show that they have some sort of localizers.

For a Kripke model {\Kr} and $ X\subseteq K $, let  
$ r_1(\kcal,X):=\# \{   \alpha\in X:  d(\alpha) \text{ is infinite }\} $ (the operator $ \#  $ counts the cardinality of its operand) and 
$ r_2(\kcal,X):={\sf max}\{d(\alpha)+1: 
\alpha\in X \text{ and } d(\alpha) \text{ is finite}  \} $ and $ r(\kcal,X):=(r_1(\kcal,X),r_2(\kcal,X)) $.  
Finally define 
$$ r(\kcal):=
 {\sf max}\{r(\kcal,X): X \text{ is a set of pairwise incomparable nodes in } \kcal\}  $$ 
 in which we use $ < $ as lexicographical order on pairs of numbers.  Through these definitions, 
 as is common, we assume that $ {\sf max}\{\}:=0 $. We say $r(\kcal)$ is finite if its both components are finite. 
Note that  $ \kcal $ is semi-narrow iff $ r(\kcal) $ is a finite  
 number.
In the above examples, the rank for ``tick'' , ``V'' and ``comb'' frames are $ (1,4) $, $ (2,0) $ 
and $ (1,1) $, respectively.

\begin{lemma}\label{lem0}
Let {\Kr} be a rooted semi-narrow Kripke model  (with $ \alpha_0$ as its root)
 for the language $ \lcal $. Also let $ \rho =A\vee \neg A $ be a 
 sentence in the language 
 $ \lcal(D(\alpha_0)) $ such that $ \kcal\nVdash \rho $.
Then $ r(\kcal^\rho )<r(\kcal) $.
\end{lemma}
\begin{proof}
One may easily prove the lemma by observing the following facts:
	\begin{enumerate}
			\item     $ X\subseteq K^\rho $  is pairwise incomparable 
		in $ \kcal ^\rho$ iff	  it is pairwise incomparable in $ \kcal $,
		\item $ d_{\kcal^\rho}(\alpha)<d_\kcal (\alpha)$, for any $ \alpha\in K^\rho $,
		\item  for any  set $ X\subseteq K $ of pairwise incomparable 
		nodes, we have $$ r(\kcal^\rho, X\cap K^\rho )< r(\kcal,X) $$
	\end{enumerate}
\end{proof}
We say that $ \alpha\in K $ is {\em weakly classical} in {\Kr} if 
$\alpha\Vdash {\sf PEM}_{\sf sen}(\lcal(D(\alpha)))$.
\begin{lemma}\label{lem1}
	Let {\Kr} be a semi-narrow Kripke model   with tree frame for the 
	language $ \lcal $. Then for any  $ \alpha\in K $, there exists some 
	$ \rho\in {\sf PEM}_{\sf sen}(\lcal(D(\alpha)))^* $ such that 
	$ \alpha $ is weakly classical in $\kcal^\rho$ . 
\end{lemma}
\begin{proof}
	 Without loss of generality, we may assume that 
	 $ \alpha  $ is the root of $ \kcal $. We use induction on $ r(\kcal) $ and prove the lemma. As induction hypothesis, assume that for any Kripke model 
	 $ \kcal_1:=(K_1,\leq_1 , D_1, \Vdash_1) $ with $ r(K_1)<n $, the lemma holds and let $ \kcal $ be a rooted Kripke model with 
	 $ r(\kcal)=n $, and $ \alpha $ as its root. If for any 
	 sentence $ A $ in the language $ \lcal(D(\alpha)) $, it holds that  
	  $ \kcal,\alpha\Vdash  A\vee \neg A$, then $ \alpha  $ is weakly classical and 
	  $ \rho:=\bot $ works. So assume that  $ \kcal,\alpha\nVdash  A\vee \neg A$, for
	  some sentence $ A\in \lcal(D(\alpha)) $.   Let $ \delta:=A\vee \neg A $.
	   By \Cref{lem0}, we have $ r(\kcal^\delta)<r(\kcal) $ and induction
	   hypothesis applies to 
	  $ \kcal^\delta $. Hence there exists some 
	  $ \theta \in {\sf PEM}_{\sf sen}(\lcal(D(\alpha)))^*$ 
	  such that $ (\kcal^\delta)^\theta $ is weakly classical at 
	  $ \alpha $. Since $   (\kcal^\delta)^\theta =  \kcal^{(\theta^\delta)}$,
	   and $ \rho := \theta^\delta\in  {\sf PEM}_{\sf sen}(\lcal(D(\alpha)))^*$,
	   we have the desired result.
\end{proof}
Let us define the translation $ ( A)^\forall  $ from \cite{AH2}.
For a formula $ A $ in an arbitrary language, let $ A^\forall  $ be the  formula obtained from $ A $ by replacing any 
$ \forall x B $ subformula of $ A $ by $ \forall x\neg \neg B $
 (This is a variant of the Kuroda translation  \cite[3.3.7]{TD}). 

The following lemma is from \cite{AH2}.
\begin{lemma}\label{lem2}
	 Let {\Kr} be a Kripke model and $ \alpha $ be a weakly classical node. 
	 Then for any sentence $ A $ in  $ \lcal(D(\alpha)) $, 
	 $ \alpha\Vdash A^\forall  $ iff $ \alpha\models A $.
\end{lemma}
\begin{proof}
	Use induction on the complexity of $ A $.
\end{proof}
\begin{theorem}\label{theorem3}
	For a  semi-narrow Kripke model  {\Kr} with tree frame for a 
	language $ \lcal $ and any  $ \alpha\in K $, there exists some 
	$ \rho\in {\sf PEM}_{\sf sen}(\lcal(D(\alpha)))^* $ such that  for all sentences 
	$ A\in \lcal(D(\alpha)) $, 
	\begin{center}
		 $ \alpha\Vdash (A^\forall)^\rho $ \quad iff \quad $ \alpha\models A $
	\end{center}
\end{theorem}
\begin{proof}
Use \Cref{lem1,lem2,pruning-lemma}.
\end{proof}

\section{Refinements}

In this section we strengthen  \Cref{main}.  We will examine the  question whether is it possible to minimize the set $ {\sf PEM} $  in \Cref{main}? In  \Cref{main-refinement}, we will show that $ {\sf PEM}_1 $ (see \Cref{Def-hieght-PEM}) is enough, however 
we do not know if $ {\sf PEM}_1 $ is the minimal set. 

Hosoi in \cite{hosoi1967intermediate} introduces slices $ \mathcal{S}_n $ for 
the intermediate logics and Ono in \cite{ono1971kripke} shows that there is 
 a tight relationship between slices and  depth-$n$ Kripke models in the following sense. 
 ``The logic of a Kripke frame is in the slice $ \mathcal{S}_n $ iff the height of the Kripke frame is $ n $”.
 
 In this paper, we use the height (depth) of Kripke models to slice the formulas in the language.

\begin{definition}\label{Def-hieght-PEM}
	Let $ \lcal  $ be an arbitrary first-order language or propositional language. The Kripke-rank of a formula 
	$ A \in\lcal $, $ h_{_\lcal}(A) $,  is the minimum  number $ n $, such that there exists some depth-$ n $ Kripke model refusting $A$, 
	$ \kcal\nVdash A $. If there is some infinite-depth Kripke model which
	 refutes $ A $ and  no finite-depth Kripke model refuting $ A $, then we define
	  $ h(A):=\omega $.
       	If  there is no   Kripke model  $ \kcal \nVdash A  $, we define 
	$ h(A):=\infty $. 
	For a set of formulas $ \Gamma\subseteq \lcal $, let  
	$ \Gamma_n:=\{A\in \Gamma : h_{_\lcal}(A)= n\}$. 
\end{definition}‌

It is clear that in any language ($ \sqcup $ means disjoint union)
$$
 {\sf PEM}_{}={\sf PEM}_{\infty}\sqcup{\sf PEM}_{\omega}\sqcup \bigsqcup_{k\in\mathbb{N}} {\sf PEM}_{k} 
$$
Since intuitionistic propositional logic has finite model property, 
there is no $ A\in \lcal_0$ with $ h(A)=\omega $. Hence 
$ {\sf PEM}_\omega=\emptyset $ in propositional language.
Before we continue, let's observe that $ {\sf PEM}_n(\lcal_0) $ is nonempty, for any $ n\in \omega $. 
Define $ A_n \in\lcal_0$  by the following clauses:
\begin{itemize}
	\item $ A_0:=\bot $,
	\item $ A_{n+1}:=p_n\vee (p_n\to A_n) $.
\end{itemize}
\begin{proposition}
For all $ n\in\mathbb{N} $, we have 	$ A_{n}\vee\neg A_n\in {\sf PEM}_{n}(\lcal_0)$.
\end{proposition}
\begin{proof}
 First we show that $ h(A_n)=n $, by induction on $n$. We note that the same Kripke model which refutes 
 $ A_n $ also refutes $ A_n\vee \neg A_n $. This implies the desired result.
\end{proof}
In the following lemma, we use the notation $ \kcal'_\alpha $ from \Cref{Def-K-alpha}, and ${\kcal'_\alpha},\alpha\Vdash A$ means that the node $\alpha$ in model ${\kcal'_\alpha}$ forces $A$.
\begin{lemma}\label{Lem-Kripke-classical}
	Let $ \alpha $  be a node of the finite-depth Kripke model  $ \kcal $. 
	Then for all $ A\in\lcal(D(\alpha)) $, 
\begin{center}	
$ {\kcal'_\alpha},\alpha\Vdash A$ iff $\kcal,\alpha\Vdash A$. 
\end{center}	
\end{lemma}
\begin{proof}
	Proof is by induction on $ d(\alpha) $.
	\begin{itemize}
		\item $ d(\alpha)=0 $.  In this case,  $ \kcal'_\alpha =\kcal_\alpha$.
		\item $ d(\alpha)=n>0 $. 
		Note that for all non-classical $ \beta>\alpha $, we have 
		$ \kcal'_\beta=(\kcal'_\alpha)_\beta $. Then  for all $ A\in\lcal(D(\beta)) $, 
\begin{center}		  
$ \kcal'_\alpha,\beta\Vdash A$ iff 
		$\kcal'_\beta,\beta\Vdash A$.
\end{center}		
		This, by induction hypothesis, implies 
\begin{center} 
$ \kcal'_\alpha,\beta\Vdash A$ iff  
		$\kcal ,\beta\Vdash A$. 		
\end{center}		
Also it's not difficult to observe that for any 
		classical node $ \beta>\alpha $, there exists some leaf
		$ \gamma\geq\beta $  (actually any leaf above $ \beta  $ works) such that for all 
		$ A\in\lcal(D(\beta)) $,  we have
\begin{center}		 
$ \kcal,\gamma\Vdash A$ iff  
		$\kcal ,\beta\Vdash A$.
\end{center}				
By use of the above mentioned facts, it is routine to prove the result by induction 
		on $ A\in\lcal(D(\alpha)) $. 
	\end{itemize}
\end{proof}

\begin{proposition}\label{ClassicalNode2}
	 For all finite-depth Kripke models \Kr, 
	 \begin{center}
	 	 $\alpha$ is a classical node  iff   $ \alpha  \Vdash {\sf PEM}_1 $.
	 \end{center}
 \end{proposition}
\begin{proof} 
	Left to right direction is deduced by \Cref{ClassicalNode}. 
	For the other way around, 
	we use induction on $ d(\alpha) $.
	\begin{itemize}
		\item $ d(\alpha)=0 $. That is obvious.
		\item $ d(\alpha)=n>0 $.  Since $ \alpha\Vdash {\sf PEM}_1 $, 
		for all $ \beta> \alpha $, we have $ \beta\Vdash {\sf PEM}_1 $, and by induction hypothesis $ \beta $ is classical node. This implies that $d( \kcal'_\alpha)=1 $ and
		hence no {\sf PEM} instance could be refuted in $ \kcal'_\alpha $  other than those which are in $ {\sf PEM}_1 $.  This implies that $ \kcal'_\alpha\Vdash {\sf PEM} $.
		\Cref{Lem-Kripke-classical} implies that $ \kcal,\alpha\Vdash {\sf PEM} $ and then  by \Cref{ClassicalNode}, we can deduce that $ \alpha $ is classical.
	\end{itemize}
\end{proof}

\begin{theorem}\label{main-refinement}
	Suppose {\Kr} is  a finite-depth Kripke model for the language $\lcal$. 
	For any   $\alpha\in K$, there exists some $\rho\in{\sf PEM}_1^*$ 
	such that for any sentence $A$ in $\lcal(D(\alpha))$,
\begin{center}	
$\alpha\Vdash A^{\rho}$ iff  $\alpha\models A$.
\end{center}
\end{theorem}
\begin{proof}
The same proof of \Cref{main} works here, by using \Cref{ClassicalNode2} and  replacing  $ {\sf PEM} $  by  $ {\sf PEM}_1 $.
\end{proof}

\section{Some applications}

Now we state some applications of \Cref{main-refinement}:

\begin{corollary}\label{Corl-main}
Let $T$ be a theory which is closed under Friedman's translation  
$(.)^\rho$, for any $\rho\in{\sf PEM}_1$. 
Then 
any finite-depth Kripke model of $T$ is locally $T$.  
\end{corollary}
\begin{proof}
First note that  by iterated use of \Cref{FriedmanAssociative}, $T$ is closed under Friedman's translation $(.)^\rho$, for any $\rho\in{\sf PEM}_1^*$. Let {\Kr} be a finite-depth Kripke model for $T$ and 
$\alpha\in K$. By 
\Cref{main-refinement}, we can find a sentence $\rho\in {\sf PEM}_1^*$ such that  for each sentence $\varphi$ in the language $\lcal(D(\alpha))$, we have 
$$\alpha\Vdash\varphi^{\rho} \Longleftrightarrow \alpha\models\varphi$$ 
Since $\alpha\Vdash T$ and $T\vdash T^\rho$, then  $\alpha\Vdash T^\rho$. Hence $\alpha\models T$.
\end{proof}
Now we can deduce a result first appeared in \cite{DKMV}:
\begin{remark}
	{\sf HA} is closed under arbitrary Friedman's translation, by proposition \ref{friedman}, hence every finite-depth  Kripke model of $\HA$ is locally {$\PA$}.
\end{remark}

\begin{corollary}
	Let $T$ be a theory over the language $ \lcal  $ 
	which is closed under the translation $ ( )^\forall $ and 
	$ {\sf PEM} $-Friedman's translation.	Then 
	any  semi-narrow Kripke model of $T$ with tree frame is locally  $T$.  
\end{corollary}
\begin{proof}
	Let $\kcal,\alpha \Vdash T $ be a semi-narrow Kripke model with tree frame. 
	By \Cref{theorem3},  there is some $ \rho\in {\sf PEM}(\lcal(D(\alpha)))^* $ such that  for all $ A\in\lcal(D(\alpha)) $, we have $ \alpha\Vdash (A^\forall)^\rho $ iff 
	$ \alpha\models A $. Since $ \alpha\Vdash T $ and $ T $ is closed under 
	$ {\sf PEM} $-Friedman's translation and 
	$ ()^\forall $ translation, we have $ \alpha\Vdash (T^\forall)^\rho $. Hence 
	$ \alpha\models T $.
\end{proof}
\begin{remark}\label{remark0}
{\em Since $ \HA $ is closed under $ {\sf PEM} $-Friedman's translation and 
$ ()^\forall $ translation (\cite{AH2}), we can deduce from the above Corollary 
that all semi-narrow Kripke models of $ \HA $ are locally $ \PA $.}
\end{remark}

For any sets $ \Gamma $ and $ \Delta $  of formulas in $ \lcal$,
let $ {\Gamma}^\Delta:=
\{A^B: A\in\Gamma,\ B\in \Delta\}$. 
Then we have:
\begin{enumerate}
	\item   $ \Gamma^{\Delta^*} $ is closed under 
	            the ${\Delta}$-Friedman's translation, i.e. for any $ A\in{\Gamma}^{\Delta^*} $ and $ B\in{\Delta} $,  $ A^B $
	            is intuitionistically equivalent to some formula in $ {\Gamma}^{\Delta^*} $,
	  \item $ \Gamma^{\Delta^*} $ is the closure of $ \Gamma $ under 
	  the ${\Delta}$-Friedman's translation, i.e.  $ \Gamma^{\Delta^*} $ is the minimum set $ X\supseteq\Gamma $ such that for all $ A\in X $ and 
	  $ B\in \Delta $ we have $ A^B\in X $,
	 \item  $i{\Gamma}^{\Delta^*}$ is closed under  ${\Delta}$-Friedman's translation.
\end{enumerate}
 The first item can be deduced easily by use of \Cref{FriedmanAssociative}, and the third item is a consequence of the first one.  Second item is straightforward.
 We have the following facts:
\begin{equation}\label{eq2}
 i\Gamma\quad \subseteq \quad i\Gamma^{\sf PEM_1^*}  
 \quad \subseteq \quad i\Gamma^{\sf PEM^*}  
 \quad \subseteq \quad   i\Gamma^{\sf Sen^*}=i\Gamma^{\sf Sen}
 \end{equation}
in which $ {\sf Sen} $ is the set of all sentences in the language of arithmetic $ \lcal_{\sf a} $.  
Since for every set $ \Gamma  $ of formulas,  $ \bot\in \Gamma^* $,  
if we let $ \Gamma $ as the set of all formulas in $ \lcal_{\sf a} $, 
all above theories are the same and equal to $ \HA $. 

\begin{question}
	 In case $ \Gamma=\Sigma_n $, $ \Gamma=\Pi_n $ or  $ \Gamma=\Phi_n $   {\em (}definition comes next{\em )}, are  the  inclusions of \cref{eq2} strict?
\end{question}

Let us recall that $\PA^-$ indicates the set of axioms for non-negative discretely ordered ring as stated in \cite{Kaye}. Let $ \Gamma $ be an arbitrary set of formulas. Then
 $\text{I}\Gamma$  ($i\Gamma$)
is the  $\vdash^c$-closure ($ \vdash $-closure) of $\PA^-$ plus
induction principle for arbitrary formulas in $ \Gamma $. $\PA$ and $\HA$ 
are $\text{I}\lcal_{\sf a}$ and  $i\lcal_{\sf a}$, respectively.
The Burr's classes $\Phi_n$ of  formulas in $\lcal_{\sf a}$ (\cite{Burr}), are defined inductively by the following items:
\begin{itemize}
	\item $ \Phi_0:=\{A\in\lcal_{\sf a}: \ A\text{  is open}\} $,
	\item $\Phi_1:= \{\exists \bar{x} \, A:\ A\in\Phi_{0} \}$,  ($ \bar{x} $ means a list of variables)
	\item $ \Phi_n:= \{\forall\,\bar{x} (B\to \exists\,\bar{y}C): \ B\in\Phi_{n-1}, \ C\in\Phi_{n-2} \} \cup \Phi_{n-1}$, for $ n\geq 2 $.
\end{itemize}
Some interesting facts about Burr's classes of formulas are  
\begin{itemize}
	\item Every formula  in $ \lcal_{\sf a} $ is equivalent (in $ i\Sigma_1 $ and even weaker theories) to a formula  in some $ \Phi_n $, 
	\item For $ n\geq 2 $, every formula in $ \Phi_n $
	is classically equivalent to some 	$ \Pi_n $ formula,
	\item For every $ n\geq 2 $, $ \text{I}\Pi_n $ is $ \Pi_2 $-conservative over 
	$ i\Phi_n $.
\end{itemize}
These properties make the Burr's fragments $i\Phi_n$ as {\em natural} fragments of $\HA$.

\begin{corollary}
Burr's hierarchies  of $\HA$, $i\Phi_n$ 
are not closed under ${\sf PEM}_1 $-Friedman's translation, i.e.   
  there exists a formula $A$
 such that $i\Phi_n\vdash A$ but 
 $i\Phi_n\nvdash A^\rho$ and $\rho\in {\sf PEM}_1$.
\end{corollary}

\begin{proof}
From \cite{Po}, we know that for each $n$, we can find a finite Kripke model for  $i\Phi_n$ such that it is not locally a model of $i\Phi_n$. Now by the previous corollary, we have the desired result.
\end{proof}

\section*{Acknowledgement}
 The author of this paper is thankful from Mohammad Ardeshir, for his valuable comments and remarks.

\providecommand{\bysame}{\leavevmode\hbox to3em{\hrulefill}\thinspace}
\providecommand{\MR}{\relax\ifhmode\unskip\space\fi MR }
\providecommand{\MRhref}[2]{%
	\href{http://www.ams.org/mathscinet-getitem?mr=#1}{#2}
}
\providecommand{\href}[2]{#2}



\end{document}